\documentclass[11pt]{article}
\usepackage{amsmath, amsfonts,amssymb, amsthm}
\usepackage[T1]{fontenc}
\usepackage[english]{babel}
\usepackage{color}
\usepackage{tikz}
\usepackage[a4paper, top=1.1in, bottom=1.35in, left=1.35in, right=1.35in]{geometry}

\usetikzlibrary{patterns}

\colorlet{pink}{red!20}
\colorlet{sky}{blue!40}
\colorlet{grass}{green!30}
\colorlet{flower}{yellow!50}
\colorlet{lightgray}{black!30}

\usepackage{enumitem}
\setlist{itemsep=2pt, topsep=2pt}

\newtheorem{theorem}{Theorem}[section]

\newtheorem{lemma}[theorem]{Lemma}

\newtheorem{proposition}[theorem]{Proposition}
\newtheorem{observation}[theorem]{Observation}
\newtheorem{fact}[theorem]{Fact}

\newtheorem{problem}[theorem]{Problem}

\newcommand{\floor}[1]{\left\lfloor#1\right\rfloor}
\newcommand{\ceiling}[1]{\left\lceil#1\right\rceil}

\newcommand{\tbf}[1]{\textbf{#1}}

\newcommand{\fair}{\floor{\frac{1}{\lambda}\ceiling{\frac{n+1}{2}}}}
\newcommand{\fairk}{\floor{\frac{\ceiling{\frac{n+1}{2}}}{\ceiling{\frac{k+1}{2}}}}}
\newcommand{\cB}{\mathcal{B}}
\newcommand{\cP}{\mathcal{P}}

\title{Hamiltonian cycles in $k$-partite graphs} 
\author{Louis DeBiasio\thanks{Department of Mathematics; Miami University; Oxford, OH.  {\tt debiasld@miamioh.edu}, {\tt kruegera@miamioh.edu}, {\tt pritikd@miamioh.edu}, {\tt thompses@miamioh.edu}} \thanks{Research supported in part by Simons Foundation Collaboration Grant \# 283194.} \and Robert A. Krueger\footnotemark[1] \and Dan Pritikin\footnotemark[1] \and Eli Thompson\footnotemark[1]}
\date{\today}

\begin{document}
\maketitle


\begin{abstract}
Chen, Faudree, Gould, Jacobson, and Lesniak \cite{CFGJL} determined the minimum degree threshold for which a balanced $k$-partite graph has a Hamiltonian cycle.  We give an asymptotically tight minimum degree condition for Hamiltonian cycles in arbitrary $k$-partite graphs in which all parts have at most $n/2$ vertices (a necessary condition). 
To do this, we first prove a general result which both simplifies the process of checking whether a graph $G$ is a robust expander and gives useful structural information in the case when $G$ is not a robust expander.  Then we use this result to prove that any $k$-partite graph satisfying the minimum degree condition is either a robust expander or else contains a Hamiltonian cycle directly.
\end{abstract}

\section{Introduction}
 
We write $A\sqcup B$ for the union of disjoint sets $A$ and $B$.  We write $G=(V_1\sqcup\dots \sqcup V_k, E)$ for a $k$-partite graph with parts $V_1, \dots, V_k$.  For $v\in V(G)$ we let $N(v)=\{u: uv\in E(G)\}$, $d(v)=|N(v)|$, and for $S\subseteq V(G)$, we let $d(v, S) = |N(v) \cap S|$. For $S\subseteq V(G)$, we let $N(S)=\{u: uv\in E(G) \text{ for some } v\in S\}$, we let $\delta(S)=\min\{d(v): v\in S\}$, and we let $\delta(R,S)=\min\{d(v, S): v\in R\}$.  

Dirac \cite{D} proved that every graph on $n\geq 3$ vertices with $\delta(G)\geq n/2$ contains a Hamiltonian cycle; furthermore, for every $n$, there exists a graph $F$ on $n$ vertices with $\delta(F)=\ceiling{n/2}-1$ such that $F$ does not contain a Hamiltonian cycle.  Later, Moon and Moser \cite{MM} proved that every balanced bipartite graph $G$ on $n\geq 4$ vertices with $\delta(G)>n/4$ contains a Hamiltonian cycle; furthermore, for every even $n$, there exists a balanced bipartite graph $F$ on $n$ vertices with $\delta(F)=\floor{n/4}$ such that $F$ does not contain a Hamiltonian cycle.  Finally, Chen et al.\ \cite{CFGJL} determined a sharp (within a constant) minimum degree condition for Hamiltonicity in all balanced $k$-partite graphs with $2\leq k\leq n$.  

\begin{theorem}[Chen, Faudree, Gould, Jacobson, Lesniak \cite{CFGJL}]\label{balancedk}
For all integers $n\geq k\geq 2$, if $G$ is a balanced $k$-partite graph on $n$ vertices such that
$$\delta(G)>\left(\frac{1}{2}+\frac{1}{2\ceiling{\frac{k+1}{2}}}-\frac{1}{k}\right)n= \begin{cases} \left(\frac{1}{2}-\frac{1}{k(k+1)}\right)n, & k \text{ odd,}   \\
\left(\frac{1}{2}-\frac{2}{k(k+2)}\right)n, & k \text{ even,}\end{cases}$$
then $G$ has a Hamiltonian cycle.  
\end{theorem}

Recently, the first author and Spanier \cite{DS} slightly refined the above result, proving the following: For all integers $n\geq k\geq 2$, if $G$ is a balanced $k$-partite graph on $n$ vertices such that 
\begin{equation}\label{eqds}
\delta(G)\geq \ceiling{\frac{n}{2}}+\fairk-\frac{n}{k},
\end{equation}
then $G$ has a Hamiltonian cycle -- unless $n$ is divisible by 4 and $k=2$ or $k=\frac{n}{2}$, in which case $\delta(G)\geq \ceiling{\frac{n}{2}}+\fairk-\frac{n}{k}+1$ suffices.  Furthermore, Proposition \ref{ml_example} (Case 4) shows that this is best possible (unless $n$ is divisible by 4 and $k=2$ or $k=\frac{n}{2}$ in which case there are different tightness examples).  For purposes of comparison, we note that 
$$\floor{\left(\frac{1}{2}+\frac{1}{2\ceiling{\frac{k+1}{2}}}-\frac{1}{k}\right)n}\leq \ceiling{\frac{n}{2}}+\fairk-\frac{n}{k}\leq \ceiling{\left(\frac{1}{2}+\frac{1}{2\ceiling{\frac{k+1}{2}}}-\frac{1}{k}\right)n}.$$
Furthermore, note that when $k=n$, Theorem \ref{balancedk} is equivalent to Dirac's theorem and when $k=2$, Theorem \ref{balancedk} is equivalent to Moon and Moser's theorem.

Regarding Hamiltonian cycles in not-necessarily-balanced $k$-partite graphs, the following (implicit) result was used as a lemma to prove a theorem about monochromatic cycles in 2-edge colored graphs of minimum degree $(3/4+o(1))n$.

\begin{proposition}[DeBiasio, Nelsen {\cite[Lemma~7.1]{DN}}]\label{3/4-}
For all $\gamma>0$, there exists $n_0$ such that for all integers $k\geq 2$ and $n\geq n_0$, if $G=(V_1\sqcup \dots \sqcup V_k, E)$ is a $k$-partite graph on $n$ vertices with $n/2\geq |V_1|\geq \dots\geq |V_k|$ such that
\[\delta(V_i)\geq \left(\frac{3}{4}+\gamma\right)n-|V_i| ~\text{ for all } i\in [k],\]
then $G$ has a Hamiltonian cycle.
\end{proposition}

We noticed that for balanced bipartite graphs and balanced tripartite graphs, the degree condition of Proposition \ref{3/4-} essentially matched the degree condition in Theorem \ref{balancedk}.  We also suspected that the degree condition of Proposition \ref{3/4-} was far from best possible in most other cases.  This motivated the following problem, which it is the purpose of this paper to address:

Determine a function $\Phi$ such that for all $\gamma>0$, there exists $n_0$ such that if $G=(V_1\sqcup \dots \sqcup V_k, E)$ is a $k$-partite graph on $n\geq n_0$ vertices with all part sizes at most $n/2$ such that
\begin{equation}\label{eq:form}
\delta(V_i)\geq \Phi(|V_1|, |V_2|, \dots, |V_k|)+\gamma n-|V_i| ~\text{ for all } i\in [k],
\end{equation}
then $G$ has a Hamiltonian cycle; furthermore, given integers $n_1, \dots, n_k$ with $\sum_{i=1}^n n_i=n$ and $n/2\geq n_1\geq \dots\geq n_k$, show that there exists a $k$-partite graph $F=(V_1\sqcup \dots \sqcup V_k, E)$ on $n$ vertices with $|V_i|=n_i$ for all $i\in [k]$ such that $\delta(V_i)\geq \Phi(|V_1|, |V_2|, \dots, |V_k|)-|V_i|-1$ for all $i\in [k]$, yet $F$ has no Hamiltonian cycle.

Note that in the case where $G$ is balanced $k$-partite graph, such a function $\Phi$ must essentially match the value given in Theorem \ref{balancedk} (more specifically the value in \eqref{eqds}).

In order to describe a function $\Phi$ satisfying the above criteria, we begin with some preliminary definitions.  For integers $1\leq k\leq n$, a \tbf{partition of $n$ into $k$ parts} is a non-increasing list $(n_1, \dots, n_k)$ of positive integers such that $\sum_{i=1}^kn_i=n$.  Let $\cP_k(n)$ be the set of all partitions of $n$ into $k$ parts.  For all $P=(n_1, n_2, \dots, n_k)\in \cP_k(n)$, let
\begin{itemize}
\item $\lambda(P)$ be the smallest positive integer $\lambda$ such that $\sum_{i=1}^{\lambda}n_i\geq \ceiling{\frac{n+1}{2}}$, and 
\item $\mu(P)$ be the smallest positive integer $\mu$ such that $\sum_{i=1}^{\mu}n_i\geq \floor{\frac{n+1}{2}-\frac{n_\mu}{2}}+1$.  
\end{itemize}
Now given $P=(n_1, n_2, \dots, n_k)\in \cP_k(n)$ with $\lambda(P)=\lambda$ and $\mu(P)=\mu$, set 
\begin{itemize}
\item $f_i(P)=n_i+\sum_{j=1}^in_j \text{ for all } i\in [\mu-1] ~\text{ and } ~f(P)=\max_{i\in [\mu-1]}f_i(P) ,$

\item $g(P)=\ceiling{\frac{n}{2}+\frac{n_\mu}{2}} ,$

\item $h_1(P)=\ceiling{\frac{n}{2}}+n_\lambda,~ h_2(P)=\ceiling{\frac{n}{2}}+\fair, \text{ and } h(P)=\min\{h_1(P), h_2(P)\} ,$

\item $\Phi(P)=\max\{f(P), g(P), h(P)\}.$
\end{itemize}

Note that that $\Phi$ does not depend on $k$, the number of parts of the partition.  Instead, $\Phi$ is a function of $|V_1|, \dots, |V_k|$ (from which we can extract the relevant information, which are the values of $\lambda=\lambda(|V_1|, \dots, |V_k|)$, $\mu=\mu(|V_1|, \dots, |V_k|)$, $|V_1|, \dots, |V_\mu|$, and $|V_\lambda|$).  

With the definition of $\Phi$ in hand, we may state our main result.  

\begin{theorem}\label{l-boundedHC}
Let $k\geq 2$ and $0< \frac{1}{n_0}\ll \gamma$.  If $G=(V_1\sqcup \dots \sqcup V_k, E)$ is a $k$-partite graph on $n \geq n_0$ vertices with $n/2\geq |V_1|\geq \dots\geq |V_k|$ such that
\[ \delta(V_i) \geq \Phi(|V_1|, |V_2|, \dots, |V_k|) +\gamma n-|V_i| ~\text{ for all } i\in [k],\]
then $G$ has a Hamiltonian cycle.
\end{theorem}

The function $\Phi$ may strike the reader as overly complicated.  However, it is in some sense necessarily so; that is, Theorem \ref{l-boundedHC} is asymptotically tight for all $k$-partite graphs on $n$ vertices.

\begin{proposition}\label{ml_example}
Let $n\geq k\geq 2$ and let $P=(n_1, \dots, n_k)$ be a partition of $n$ into $k$ parts with $n/2\geq n_1\geq \dots \geq n_1$.  There exists a $k$-partite graph $F=(V_1\sqcup \dots \sqcup V_k, E)$ with $|V_i|=n_i$ for all $i\in [k]$ such that
\[ \delta(V_i) \geq \Phi(|V_1|, |V_2|, \dots, |V_k|) - |V_i| - 1 ~\text{ for all } i\in [k],\]
yet $F$ has no Hamiltonian cycle.
\end{proposition}

We will prove Proposition \ref{ml_example} in Section \ref{sec:examp}, and we will prove Theorem \ref{l-boundedHC} in Section \ref{sec:robust}.  In Section \ref{sec:frac}, we prove a relaxed version of Theorem \ref{l-boundedHC} (but with a tight minimum degree condition) in order to motivate the methods used in Section \ref{sec:robust}.

\subsection{Open problems and further observations} 

Proving a non-asymptotic version of Theorem \ref{l-boundedHC} is the most immediate remaining problem.

\begin{problem}\label{tightconjec}
For all integers $n\geq k\geq 2$, let $G=(V_1\sqcup \dots \sqcup V_k, E)$ be a $k$-partite graph on $n$ vertices with $n/2\geq |V_1|\geq \dots\geq |V_k|$.  Determine the smallest constant $c$, depending only on $|V_1|, \dots, |V_k|$, such that if  
\[ \delta(V_i) \geq \Phi(|V_1|, |V_2|, \dots, |V_k|)+c-|V_i| ~\text{ for all } i\in [k],\]
then $G$ has a Hamiltonian cycle.
\end{problem}

It is possible that aside from some exceptional cases (e.g. the results from \cite{DS} mentioned earlier), the constant could be 0.  This possibility is further supported by the upcoming Theorem \ref{warmup}.

We also raise the following problem regarding a possible alternative form of the minimum degree condition \eqref{eq:form}.

\begin{problem}\label{prob2}
Determine a function $\Psi$ such that for all $\gamma>0$, there exists $n_0$ such that if $G=(V_1\sqcup \dots \sqcup V_k, E)$ is a $k$-partite graph on $n\geq n_0$ vertices with all part sizes at most $n/2$ such that
\begin{equation*}\label{eq:form2}
\delta(V_i)\geq \Psi(|V_1|, |V_2|, \dots, |V_k|)(n-|V_i|)+\gamma n  ~\text{ for all } i\in [k],
\end{equation*}
then $G$ has a Hamiltonian cycle; furthermore, given integers $n_1, \dots, n_k$ with $\sum_{i=1}^n n_i=n$ and $n/2\geq n_1\geq \dots\geq n_k$, show that there exists a $k$-partite graph $F=(V_1\sqcup \dots \sqcup V_k, E)$ on $n$ vertices with $|V_i|=n_i$ for all $i\in [k]$ such that $\delta(V_i)\geq \Psi(|V_1|, |V_2|, \dots, |V_k|)(n-|V_i|)-1$ for all $i\in [k]$, yet $F$ has no Hamiltonian cycle.  
\end{problem}

When it is unambiguous, we write $\Phi$ instead of $\Phi(|V_1|, \dots, |V_k|)$, $f$ instead of $f(|V_1|, \dots, |V_k|)$, etc.  Before moving on, we make some useful observations regarding $\Phi$ and the associated parameters.

\begin{observation}\label{l-fairfacts}
Let $k\geq 2$ and let $G=(V_1\sqcup \dots \sqcup V_k, E)$ be a $k$-partite graph on $n$ vertices with $|V_1|\geq \dots\geq |V_k|$. Then the following hold:
\begin{enumerate}[label=(\roman*)]
\item\label{1} $\lambda \leq \ceiling{\frac{k+1}{2}}$.  Furthermore, if $G$ is balanced, then $\lambda=\ceiling{\frac{k+1}{2}}$.
\item\label{2} For all $I \subseteq [k]$, if $|I| \leq \lambda-1$, then $\sum_{i \in I} |V_i| \leq \frac{n}{2}$.  So in particular, $\lambda\geq 2$ if and only if all parts have at most $n/2$ vertices. 
\item\label{4} $f_i \leq \floor{\frac{n+1}{2} + \frac{|V_i|}{2}}$ for all $i \in [\mu-1]$.
\item\label{4.5} $\sum_{i=1}^\mu |V_i| \leq \floor{\frac{n+1}{2} + \frac{|V_\mu|}{2}}$.
\item\label{7} If $\lambda \geq 2$, then $\Phi-|V_k|\geq \dots \geq \Phi-|V_1| \geq n/4$.
\item\label{8} If $G$ is balanced, then $\Phi=h_2=\ceiling{\frac{n}{2}}+\fair=\ceiling{\frac{n}{2}}+\floor{\frac{\ceiling{\frac{n+1}{2}}}{\ceiling{\frac{k+1}{2}}}}$.
\end{enumerate}
\end{observation}

\begin{proof}

\begin{enumerate}
\item When $k$ is odd, we have $\sum_{i=1}^{\frac{k+1}{2}}|V_i|>n/2$. When $k$ is even, we have $\sum_{i=1}^{\frac{k}{2}}|V_i|>n/2$ unless $G$ is balanced, in which case $\sum_{i=1}^{\frac{k}{2}+1}|V_i|>n/2$.  So $\lambda\leq \ceiling{\frac{k+1}{2}}$.  If $G$ is balanced, then clearly $\lambda\geq \ceiling{\frac{k+1}{2}}$.
\item Let $I \subseteq [k]$ with $|I| \leq \lambda-1$.  Since $|V_1|\geq\dots \geq |V_k|$, we have by the minimality of $\lambda$ that $\sum_{i \in I} |V_i| \leq \sum_{i=1}^{\lambda-1}|V_i|\leq \frac{n}{2}$.  So if $\lambda\geq 2$, then $n/2\geq |V_1|\geq \dots\geq |V_k|$, and if $\lambda=1$, then $|V_1|>n/2$.
\item Let $i\in [\mu-1]$.  By the definition of $\mu$ we have  $\sum_{j=1}^i |V_j| \leq \floor{\frac{n+1}{2} - \frac{|V_i|}{2}}$, so 
\[f_i = |V_i| + \sum_{j=1}^i |V_j| \leq |V_i|+\floor{\frac{n+1}{2} - \frac{|V_i|}{2}}= \floor{\frac{n+1}{2} + \frac{|V_i|}{2}}.\]
\item By the definition of $\mu$ we have 
\[ \sum_{i=1}^\mu |V_i| = \sum_{i=1}^{\mu-1} |V_i| + |V_\mu| \leq \floor{\frac{n+1}{2} - \frac{|V_{\mu-1}|}{2}} + |V_\mu| \leq \floor{\frac{n+1}{2} + \frac{|V_\mu|}{2}} .\]
\item If $\mu=1$ or $|V_1| < n/4$, we have $g - |V_1| = \ceiling{\frac{n}{2} + \frac{|V_\mu|}{2}} - |V_1| \geq n/4$. Otherwise $\mu>1$ and $|V_1| \geq n/4$ so $\Phi-|V_1|\geq f_1 - |V_1| = |V_1| \geq n/4$.
\item When $G$ is balanced, there exists an integer $m$ such that $n=mk$.  By \ref{1}, \ref{4}, and the definitions of $g$ and $h$, we have  
\begin{align*}
\Phi&=\max\left\{\floor{\frac{n+1}{2}+\frac{n}{2k}}, \ceiling{\frac{n}{2}+\frac{n}{2k}}, \ceiling{\frac{n}{2}}+\floor{\frac{\ceiling{\frac{n+1}{2}}}{\ceiling{\frac{k+1}{2}}}}\right\}\\
&=\max\left\{\floor{\frac{mk+1+m}{2}}, \ceiling{\frac{mk+m}{2}}, \ceiling{\frac{mk}{2}}+\floor{\frac{\ceiling{\frac{mk+1}{2}}}{\ceiling{\frac{k+1}{2}}}}\right\}.
\end{align*}
Since $\floor{\frac{mk+1+m}{2}}=\ceiling{\frac{mk+m}{2}}$, it remains to show 
\begin{equation}\label{mk}
\ceiling{\frac{mk}{2}}+\floor{\frac{\ceiling{\frac{mk+1}{2}}}{\ceiling{\frac{k+1}{2}}}}\geq \ceiling{\frac{mk+m}{2}}.
\end{equation}

If $m$ is even or $k$ is even, \eqref{mk} reduces to $\floor{\frac{\frac{mk+2}{2}}{\ceiling{\frac{k+1}{2}}}}\geq \ceiling{\frac{m}{2}}$.  Since $m\geq 1$ and $k\geq 2$, we have 
$$\frac{mk+2}{2}-\ceiling{\frac{m}{2}}\ceiling{\frac{k+1}{2}}\geq \frac{mk+2}{2}-\left(\frac{m+1}{2}\right)\left(\frac{k+2}{2}\right)=\frac{(m-1)(k-2)}{4}\geq 0,$$
and thus \eqref{mk} is satisfied.  

If $m$ is odd and $k$ is odd, \eqref{mk} reduces to $\floor{\frac{mk+1}{k+1}}\geq \frac{m-1}{2}$.  Since $m\geq 1$ and $k\geq 2$, we have $$mk+1-(k+1)\frac{m-1}{2}=\frac{(m+1)(k-1)+4}{2}>0,$$
and thus \eqref{mk} is satisfied.
\end{enumerate}
\end{proof}

Finally, note that Theorem \ref{l-boundedHC} asymptotically implies Theorem \ref{balancedk} because when $G$ is a balanced $k$-partite graph on $n$ vertices, we have by Observation \ref{l-fairfacts}.\ref{8} that 
$$\Phi(n/k, \dots, n/k)-|V_i| =\ceiling{\frac{n}{2}}+\fairk - \frac{n}{k}.
$$

Despite the fact that our result only implies Theorem \ref{balancedk} (and hence Dirac's theorem) asymptotically, we note that there are situations (other than the explicit setting of $k$-partite graphs) where Theorem \ref{l-boundedHC} can be applied to guarantee the existence of a Hamiltonian cycle where Theorem \ref{balancedk} (and hence Dirac's theorem) does not apply.  For example, suppose $G$ is a graph on $n$ vertices with $\delta(G)\geq (59/120+\gamma)n$ and suppose that there is a partition of $V(G)$ into independent sets $V_1, V_2, V_3, V_4, V_5$ with $|V_1|=|V_2|=9n/40$, $|V_3|=8n/40$, and $|V_4|=|V_5|=7n/40$.  We have $\Phi(9n/40, 9n/40, 8n/40, 7n/40, 7n/40)=2n/3$ and note for all $i\in [5]$ that $\delta(V_i)\geq (59/120+\gamma)n= 2n/3+\gamma n-|V_5|\geq 2n/3+\gamma n-|V_i|$.  Thus $G$ has a Hamiltonian cycle by Theorem \ref{l-boundedHC}.

\section{Tightness examples}\label{sec:examp}

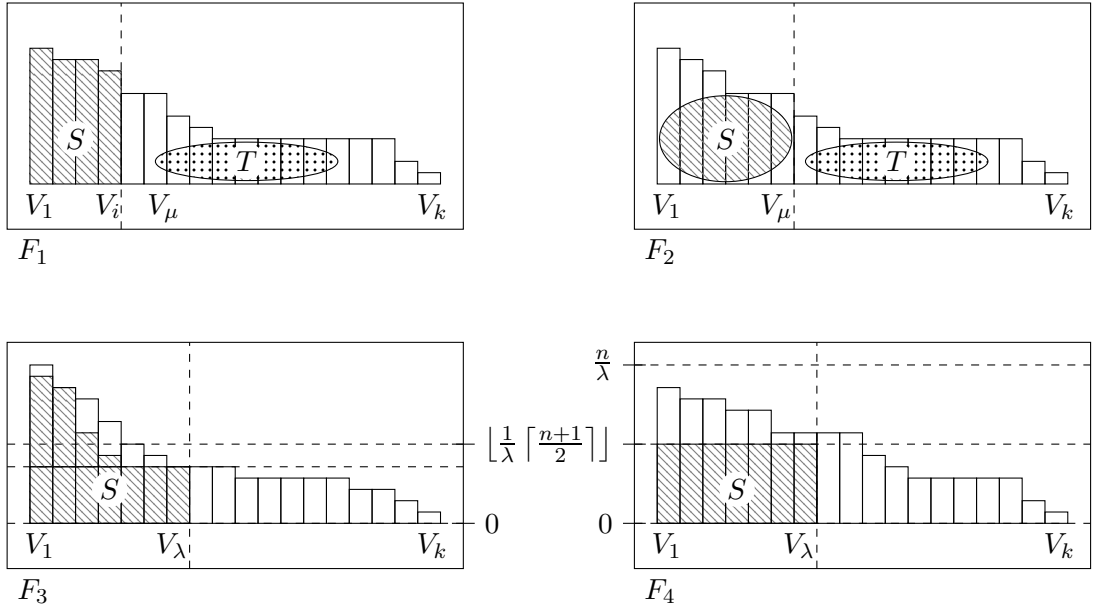
\begin{figure}[ht]
\centering
\begin{tikzpicture}[scale = .15]
\begin{scope}[shift={(0,0)}]
\draw (-20, -10) node[below right]{$F_1$} rectangle (20, 10);
\draw[dashed] (-10,-10) -- (-10, 10);
\draw[pattern=north west lines, pattern color=gray] (-18, -6) node[below]{$~~V_1$} rectangle (-16, 6);
\draw[pattern=north west lines, pattern color=gray] (-16, -6) rectangle (-14, 5);
\draw[pattern=north west lines, pattern color=gray] (-14, -6) rectangle (-12, 5);
\draw[pattern=north west lines, pattern color=gray] (-12, -6) node[below]{$~~V_i$} rectangle (-10, 4);
\draw (-10, -6) rectangle (-8, 2);
\draw (-8, -6) node[below]{$~~~~V_\mu$} rectangle (-6, 2);
\draw (-6, -6) rectangle (-4, 0);
\draw (-4, -6) rectangle (-2, -1);
\draw (-2, -6) rectangle (0, -2);
\draw (0, -6) rectangle (2, -2);
\draw (2, -6) rectangle (4, -2);
\draw (4, -6) rectangle (6, -2);
\draw (6, -6) rectangle (8, -2);
\draw (8, -6) rectangle (10, -2);
\draw (10, -6) rectangle (12, -2);
\draw (12, -6) rectangle (14, -2);
\draw (14, -6) rectangle (16, -4);
\draw (16, -6) node[below]{$~~~V_k$} rectangle (18, -5);
\draw[pattern=dots, pattern color=black] (1,-4) node[circle, fill=white, inner sep=1]{$T$} ellipse (8 and 1.7);
\draw (-14, -2) node[circle, fill=white, inner sep=1]{$S$};
\end{scope}

\begin{scope}[shift={(55,0)}]
\draw (-20, -10) node[below right]{$F_2$} rectangle (20, 10);
\draw[dashed] (-6,-10) -- (-6, 10);
\draw (-18, -6) node[below]{$~~V_1$} rectangle (-16, 6);
\draw (-16, -6) rectangle (-14, 5);
\draw (-14, -6) rectangle (-12, 4);
\draw (-12, -6) rectangle (-10, 2);
\draw (-10, -6) rectangle (-8, 2);
\draw (-8, -6) node[below]{$~V_\mu$} rectangle (-6, 2);
\draw[pattern=north west lines, pattern color=gray] (-12, -2) node[circle, fill=white, inner sep=1]{$S$} ellipse (5.8 and 3.8);
\draw (-6, -6) rectangle (-4, 0);
\draw (-4, -6) rectangle (-2, -1);
\draw (-2, -6) rectangle (0, -2);
\draw (0, -6) rectangle (2, -2);
\draw (2, -6) rectangle (4, -2);
\draw (4, -6) rectangle (6, -2);
\draw (6, -6) rectangle (8, -2);
\draw (8, -6) rectangle (10, -2);
\draw (10, -6) rectangle (12, -2);
\draw (12, -6) rectangle (14, -2);
\draw (14, -6) rectangle (16, -4);
\draw (16, -6) node[below]{$~~~V_k$} rectangle (18, -5);
\draw[pattern=dots, pattern color=black] (3,-4) node[circle, fill=white, inner sep=1]{$T$} ellipse (8 and 1.7);
\end{scope}

\begin{scope}[shift={(0,-30)}]
\draw (-20, -10) node[below right]{$F_3$} rectangle (20, 10);
\draw[pattern=north west lines, pattern color=gray] (-18, -6) rectangle (-4, -1);
\draw (21, -6) node[right]{$0$} -- (19, -6);
\draw[dashed] (-20, -6) -- (20, -6);
\draw (21, 1) -- (19, 1);
\draw[dashed] (-20, 1) -- (20, 1);
\draw[dashed] (-20, -1) -- (20, -1);
\draw[dashed] (-4,-10) -- (-4, 10);
\draw (-18, -6) node[below]{$~~V_1$} rectangle (-16, 8);
\draw[pattern=north west lines, pattern color=gray] (-18,-1) rectangle (-16,7);
\draw (-16, -6) rectangle (-14, 6);
\draw[pattern=north west lines, pattern color=gray] (-16,-1) rectangle (-14,6);
\draw (-14, -6) rectangle (-12, 5);
\draw[pattern=north west lines, pattern color=gray] (-14,-1) rectangle (-12,2);
\draw (-12, -6) rectangle (-10, 3);
\draw[pattern=north west lines, pattern color=gray] (-12,-1) rectangle (-10,0);
\draw (-10, -6) rectangle (-8, 1);
\draw (-8, -6) rectangle (-6, 0);
\draw (-6, -6) node[below]{$~V_{\lambda}$} rectangle (-4, -1);
\draw (-4, -6) rectangle (-2, -1);
\draw (-2, -6) rectangle (0, -1);
\draw (0, -6) rectangle (2, -2);
\draw (2, -6) rectangle (4, -2);
\draw (4, -6) rectangle (6, -2);
\draw (6, -6) rectangle (8, -2);
\draw (8, -6) rectangle (10, -2);
\draw (10, -6) rectangle (12, -3);
\draw (12, -6) rectangle (14, -3);
\draw (14, -6) rectangle (16, -4);
\draw (16, -6) node[below]{$~~~V_k$} rectangle (18, -5);
\draw (-11, -3) node[circle, fill=white, inner sep=1]{$S$};
\end{scope}

\begin{scope}[shift={(55,-30)}]
\draw (-20, -10) node[below right]{$F_4$} rectangle (20, 10);
\draw[pattern=north west lines, pattern color=gray] (-18, -6) rectangle (-4, 1);
\draw (-21, 8) node[left]{$\frac{n}{\lambda}$} -- (-19, 8);
\draw[dashed] (-20, 8) -- (20, 8);
\draw (-21, -6) node[left]{$0$} -- (-19, -6);
\draw[dashed] (-20, -6) -- (20, -6);
\draw (-21, 1) node[left]{$\fair$} -- (-19, 1);
\draw[dashed] (-20, 1) -- (20, 1);
\draw[dashed] (-4,-10) -- (-4, 10);
\draw (-18, -6) node[below]{$~~V_1$} rectangle (-16, 6);
\draw (-16, -6) rectangle (-14, 5);
\draw (-14, -6) rectangle (-12, 5);
\draw (-12, -6) rectangle (-10, 4);
\draw (-10, -6) rectangle (-8, 4);
\draw (-8, -6) rectangle (-6, 2);
\draw (-6, -6) node[below]{$~V_{\lambda}$} rectangle (-4, 2);
\draw (-4, -6) rectangle (-2, 2);
\draw (-2, -6) rectangle (0, 2);
\draw (0, -6) rectangle (2, 0);
\draw (2, -6) rectangle (4, -1);
\draw (4, -6) rectangle (6, -2);
\draw (6, -6) rectangle (8, -2);
\draw (8, -6) rectangle (10, -2);
\draw (10, -6) rectangle (12, -2);
\draw (12, -6) rectangle (14, -2);
\draw (14, -6) rectangle (16, -4);
\draw (16, -6) node[below]{$~~~V_k$} rectangle (18, -5);
\draw (-11, -3) node[circle, fill=white, inner sep=1]{$S$};
\end{scope}
\end{tikzpicture}
\caption{The four families of tightness examples $F_1, F_2, F_3, F_4$ in Proposition \ref{ml_example}.}
\end{figure}

\begin{proof}[Proof of Proposition \ref{ml_example}]
We will construct a $k$-partite graph on vertex set $V_1\sqcup \dots \sqcup V_k$ where $|V_i|=n_i$ for all $i\in [k]$.  Note that when we use the phrase \emph{possible edges}, we mean edges having endpoints in different parts.  The construction splits into cases depending on the value of $\Phi$.

\tbf{Case 1} ($\Phi=f_i$ for some $i\in [\mu-1]$).  Let $S=V_1\cup\dots\cup V_i$ and let $T\subseteq V_{i+1}\cup \dots\cup V_k$ with $|T|=|S|-1$.  Now add all edges between $S$ and $T$ and add all other possible edges which are not incident with $S$.  Let $F_1$ be the resulting $k$-partite graph.  Since $S$ is an independent set with $|N(S)|\leq |S|-1$, we see that $F_1$ has no Hamiltonian cycle. 

First note that since $i<\mu$, we have 
\begin{equation}\label{Sup}
|S| = \sum_{h=1}^i|V_h|\leq \floor{\frac{n+1}{2}-\frac{|V_i|}{2}}
\end{equation}
For $j\in [i]$, we have $|V_j|\geq |V_i|$ and thus
\[\delta(V_j) = |S|-1=\sum_{h=1}^i|V_h|-1\geq |V_i|+\sum_{h=1}^i|V_h|-|V_j|-1 = f_i - |V_j| - 1 = \Phi - |V_j| - 1\]
and for $j\in [k]\setminus [i]$, by \eqref{Sup} and Observation \ref{l-fairfacts}.\ref{4} we have 
\begin{align*}
\delta(V_j)\geq n-|V_j|-|S|\geq n-\floor{\frac{n+1}{2}-\frac{|V_{i}|}{2}}-|V_j|&\geq \floor{\frac{n+1}{2}+\frac{|V_{i}|}{2}}-|V_j|-1\\ 
&\geq f_i-|V_j|-1 = \Phi - |V_j|-1.
\end{align*}

\tbf{Case 2} ($\Phi=g$). By the definition of $\mu$ we have $\sum_{i=1}^\mu|V_i|\geq \floor{\frac{n+1}{2}-\frac{|V_\mu|}{2}}+1$ and by Observation \ref{l-fairfacts}.\ref{4.5} we have $\sum_{i=\mu+1}^k|V_i|=n-\sum_{i=1}^\mu|V_i|\geq n-\floor{\frac{n+1}{2} + \frac{|V_\mu|}{2}}= \floor{\frac{n}{2}-\frac{|V_\mu|}{2}}$.  So we let $S\subseteq V_1\cup\dots \cup V_\mu$ with $|S|=\floor{\frac{n}{2}-\frac{|V_{\mu}|}{2}}+1$ and $T\subseteq V_{\mu+1}\cup \dots \cup V_k$ with $|T|=|S|-1$.  Now add all edges between $S$ and $T$ and add all other possible edges which are not incident with $S$.  Let $F_2$ be the resulting $k$-partite graph.  Since $S$ is an independent set with $|N(S)|\leq |S|-1$, we see that $F_2$ has no Hamiltonian cycle.  

For $i\in [\mu]$ we have $|V_i|\geq |V_\mu|$ and thus
\[\delta(V_i)\geq |S|-1 =\floor{\frac{n}{2}-\frac{|V_{\mu}|}{2}}\geq \ceiling{\frac{n}{2}+\frac{|V_{\mu}|}{2}}-|V_i|-1 = g - |V_i| - 1 = \Phi - |V_i| - 1 \]
and for $j\in [k]\setminus [\mu]$, we have 
\begin{align*}
\delta(V_j)\geq n-|V_j|-|S| =n-\floor{\frac{n}{2}-\frac{|V_{\mu}|}{2}}-|V_j| -1&= \ceiling{\frac{n}{2}+\frac{|V_{\mu}|}{2}}-|V_j| -1\\
&= g - |V_j| - 1 = \Phi - |V_j| - 1 .
\end{align*}

\tbf{Case 3} ($\Phi=h_1$). Note that the case implies $n_\lambda \leq \floor{\frac{1}{\lambda}\ceiling{\frac{n+1}{2}}}$.  For $i \in [\lambda]$ select $X_i \subseteq V_i$ with $|X_i| \geq |V_\lambda|$ and $|X_1 \cup \dots \cup X_\lambda| = \ceiling{\frac{n+1}{2}}$. This is possible because 
\[ \lambda|V_\lambda|\leq \lambda\fair\leq \ceiling{\frac{n+1}{2}}\leq \sum_{i=1}^{\lambda} |V_i| .\]

Let $S = X_1\cup\dots \cup X_\lambda$. Now add all possible edges which have at most one endpoint in $S$. Let $F_3$ be the resulting $k$-partite graph.  Since $S$ is an independent set of size greater than $n/2$, we see that $F_3$ has no Hamiltonian cycle.  

For $i \in [k] \setminus [\lambda]$, we have $\delta(V_i) = n - |V_i| \geq \Phi - |V_i|$.  For $i\in [\lambda]$,  we have
\[\delta(V_i) = n - |S| - |V_i| + |X_i| \geq n - \ceiling{\frac{n+1}{2}} + |V_\lambda| - |V_i|=\ceiling{\frac{n}{2}}+|V_\lambda|-|V_i|-1= h - |V_i| - 1 .\]

\tbf{Case 4} ($\Phi=h_2$). Note that the case implies $n_\lambda \geq \floor{\frac{1}{\lambda}\ceiling{\frac{n+1}{2}}}$.  For $i \in [\lambda]$ select $X_i \subseteq V_i$ with $|X_i| \geq \floor{\frac{1}{\lambda}\ceiling{\frac{n+1}{2}}}$ and $|X_1 \cup \dots \cup X_\lambda| = \ceiling{\frac{n+1}{2}}$. This is possible because $|V_\lambda| \geq \floor{\frac{1}{\lambda}\ceiling{\frac{n+1}{2}}}$, $\sum_{i=1}^{\lambda} |V_i| \geq \ceiling{\frac{n+1}{2}}$, and $\lambda \floor{\frac{1}{\lambda} \ceiling{\frac{n+1}{2}}} \leq \ceiling{\frac{n+1}{2}}$. Note that there is some $j$ such that $|X_j| = \floor{\frac{1}{\lambda}\ceiling{\frac{n+1}{2}}}$, as otherwise
\[ \ceiling{\frac{n+1}{2}}  = \sum_{i=1}^{\lambda} |X_i| \geq \lambda \left( \floor{\frac{1}{\lambda}\ceiling{\frac{n+1}{2}}} + 1 \right) \geq \lambda \left( \frac{\ceiling{\frac{n+1}{2}}-(\lambda-1)}{\lambda}+ 1 \right) > \ceiling{\frac{n+1}{2}} ,\]
a contradiction. Let $S = X_1\cup\dots \cup X_\lambda$. Now add all possible edges which have at most one endpoint in $S$. Let $F_4$ be the resulting $k$-partite graph.  Since $S$ is an independent set of size greater than $n/2$, we see that $F_4$ has no Hamiltonian cycle.

For $i \in [k] \setminus [\lambda]$, we have $\delta(V_i) = n - |V_i| \geq \Phi - |V_i|$.  For $i\in [\lambda]$,  we have
$$\delta(V_i) \geq n-|V_i|-\sum_{\stackrel{j=1}{j\neq i}}^\lambda|X_j|=n-\ceiling{\frac{n+1}{2}}+|X_i|-|V_i|\geq \ceiling{\frac{n}{2}}+\fair-|V_i|-1,$$
with equality whenever $|X_j|=\fair$. 
\end{proof}

\section{Weak expansion and perfect fractional matchings}\label{sec:frac}

The purpose of this section is to draw parallels between an old result of Tutte and Lemma \ref{robustorindep} which gives a sufficient condition for a graph to be a robust expander.

For the purposes of this paper, we define a \tbf{fractional matching} to be a graph in which each component is an edge or an odd cycle (in the literature, this is sometimes referred to as a \emph{basic 2-matching}).  We say that a graph has a \tbf{perfect fractional matching} if it has a fractional matching as a spanning subgraph.  We say that $T\subseteq V(G)$ \tbf{weakly expands} if $|N(T)|\geq |T|$.  Note that if a graph $G$ has a perfect fractional matching, then every set $T\subseteq V(G)$ weakly expands.  Tutte \cite{T} proved the following, which characterizes graphs having a perfect fractional matching (note the parallels between Theorem \ref{tutte} and the upcoming Theorem \ref{robustexpanderHC}).

\begin{theorem}[Tutte \cite{T}]\label{tutte}
Let $G$ be a graph.  If $|N(T)|\geq |T|$ for all sets $T\subseteq V(G)$, then $G$ has a perfect fractional matching.
\end{theorem}

\begin{proof}
Suppose $|N(T)|\geq |T|$ for all sets $T\subseteq V(G)$.
Now let $H$ be an auxiliary bipartite graph on $V$ and $V'$, two copies of $V(G)$ (where for each $x \in V$, we denote the copy of $x$ in $V'$ as $x'$) where $xy' \in E(H)$ if and only if $xy \in E(G)$.
Note that by our assumption, Hall's condition holds and thus $H$ contains a perfect matching.  The perfect matching in $H$ corresponds to a spanning subgraph $M\subseteq G$ in which each component is an edge or a cycle.  Since any   cycle of $M$ with length $2k$ contains $k$ disjoint edges, $G$ has a perfect fractional matching.
\end{proof}

If one wants to apply Theorem \ref{tutte}, it is very useful to use the following fact which says that if all independent sets weakly expand, then all sets weakly expand.  We provide the proof for completeness and to draw parallels to the upcoming Lemma \ref{robustorindep}.

\begin{fact}\label{lem:indall}
Let $G$ be a graph on vertex set $V$.  If $|N(S)|\ge|S|$ for all independent sets $S\subseteq V$, then $|N(T)|\ge|T|$ for all sets $T\subseteq V$.  
\end{fact}

\begin{proof}
Suppose there exists $T\subseteq V(G)$ such that $|N(T)|<|T|$. Let $S$ be the set of isolated vertices in $G[T]$, the graph induced by $T$.  Note that $N(S) \subseteq N(T) \setminus (T \setminus S)$, $T \setminus S \subseteq N(T)$, and $S$ is an independent set in $G$.  Thus $|N(S)| \leq |N(T) \setminus (T \setminus S)| = |N(T)| - |T \setminus S| < |T| - |T \setminus S| = |S|$, contradicting the assumption that $|N(S)|\geq |S|$.
\end{proof}

Now we use Theorem \ref{tutte} and Fact \ref{lem:indall} to prove the following result whose main purpose is to provide a ``template'' for the upcoming proof of Theorem \ref{degcondrobust}.  Note that the examples given in Proposition \ref{ml_example} show that the degree condition in the following result is tight.

\begin{theorem}\label{warmup}
Let $n\geq k\geq 2$.  If $G=(V_1\sqcup \dots \sqcup V_k, E)$ is a $k$-partite graph on $n$ vertices with $n/2\geq |V_1|\geq \dots\geq |V_k|$ such that
\[ \delta(V_i) \geq \Phi(|V_1|, \dots, |V_k|) - |V_i| ~\text{ for all } i\in [k],\]
then $|N(S)|\geq |S|$ for all independent sets $S\subseteq V(G)$ and consequently $G$ has a perfect fractional matching.
\end{theorem}

\begin{proof}
Let $S$ be a non-empty independent set in $G$. Relabel the parts as $U_1, \dots, U_k$ and define $r$ so that $A:= U_1\cup \dots \cup U_r$ is the union of the parts which intersect $S$ and $|U_1|\geq \dots \geq |U_r|$, and $B:=U_{r+1}\cup \dots \cup U_k$ is the union of the parts which do not intersect $S$ and $|U_{r+1}|\geq \dots \geq |U_k|$.  

If $1\leq r\leq \mu-1$, then $$|N(S)|\geq \delta(U_r)\geq \delta(V_r) \geq f_r - |V_r| = \sum_{i=1}^r|V_i|\geq \sum_{i=1}^r |U_i| \geq |S|.$$

So for the remainder of the proof we suppose $r\geq \mu$, which implies $$|N(S)|\geq \delta(U_r)\geq \delta(V_\mu)\geq g - |V_\mu|.$$
If $|S|\leq g-|V_\mu|\leq |N(S)|$, then we are done; so suppose that
\begin{equation}\label{Sbound}
|S|\geq g-|V_\mu|+1.
\end{equation}

Let $v\in B$. By \eqref{Sbound} we have
\begin{align*}
d(v, S)\geq g-|B|-(|A|-|S|)=|S|-(n-g)&\geq g-|V_\mu|+1-n+g\\
&=2\ceiling{\frac{n}{2}+\frac{|V_\mu|}{2}}-(n+|V_\mu|)+1\geq 1
\end{align*}
which implies
\begin{equation}\label{BNS1}
B\subseteq N(S).
\end{equation}

\textbf{Case 1} ($\mu\leq r\leq \lambda-1$). Using $S\subseteq A$ and \eqref{BNS1}, we have 
\[|S|\leq |A| = \sum_{i=1}^r|U_i| \leq \frac{n}{2}\leq |B|\leq |N(S)|,\]
where the second inequality holds by Observation \ref{l-fairfacts}.\ref{2} (since $r\leq \lambda-1$) and the third inequality holds because $|A| + |B| = n$ and $|A| \leq \frac{n}{2}$.

\textbf{Case 2} ($r\geq \lambda$).  
If $h = h_1$, then $|N(S)|\geq \delta(U_r)\geq \delta(V_\lambda)\geq h_1 - |V_\lambda| = \ceiling{n/2}$, which (since $S$ is independent), implies that $|N(S)|\geq |S|$.  So suppose $h=h_2$, i.e.\ $|V_\lambda| \geq \fair$.

For all $v\in S\cap U_r$ we have, 
\begin{equation}\label{rearrange}
d(v, A)\geq h_2-|U_r|-|B|> \sum_{i=1}^{r-1}\left(|U_i|-\fair-1\right),
\end{equation}
where the last inequality holds since 
\begin{align*}
h_2+(r-1)\left(\fair+1\right)&\geq h_2+(\lambda-1)\left(\fair+1\right)\\
&=\ceiling{\frac{n}{2}}+\lambda\fair+(\lambda-1)\\
&\geq \ceiling{\frac{n}{2}}+\lambda\left(\frac{\ceiling{\frac{n+1}{2}}-(\lambda-1)}{\lambda}\right)+(\lambda-1)\\
&=\ceiling{\frac{n}{2}}+\ceiling{\frac{n+1}{2}}>n=\sum_{i=1}^r|U_i|+|B|.
\end{align*}

So let $v_r\in U_r\cap S$, and note that by \eqref{rearrange} there exists some $i$ with $1\leq i\leq r-1$ such that $|N(v_r)\cap U_i|\geq |U_i|-\fair$.  Letting $v_i\in U_i\cap S$ now gives 
\begin{align*}
|N(S)|\geq |N(v_1)\cap U_i|+|N(v_i)|\geq |U_i|-\fair + h_2-|U_i|=\ceiling{\frac{n}{2}},
\end{align*}
which (since $S$ is independent) implies that $|N(S)|\geq |S|$.  
\end{proof}

\section{Robust expansion and Hamiltonian cycles}\label{sec:robust}

Let $0<\nu\leq \tau/2$ and let $G$ be a graph on $n$ vertices.  For all $S\subseteq V(G)$, let $RN_\nu(S)=\{v: d(v, S)\geq \nu n\}$ be the $\nu$-\tbf{robust neighborhood} of $S$.  If $|RN_{\nu}(S)| \geq |S| + \nu n$, we say $S$ $\nu$-\tbf{robustly expands}.  We say that $G$ is a \tbf{$(\nu, \tau)$-robust expander} if $|RN_\nu(S)|\geq |S|+\nu n$ for all $S\subseteq V(G)$ with $\tau n\leq |S|\leq (1-\tau)n$.

K\"uhn, Osthus, and Treglown \cite{KOT} proved that for sufficiently large $n$, if $G$ is a robust expander on $n$ vertices with linear mimimum degree, then $G$ has a Hamiltonian cycle.  Later this result was proved by Lo and Patel \cite{LP} without the use of Szemer\'edi's regularity lemma, thus yielding better explicit bounds on the parameters.  Note that both of these results were actually proved for directed graphs, but here we only state the result for  symmetric digraphs (i.e.\ graphs).  

\begin{theorem}[Lo, Patel \cite{LP}]\label{robustexpanderHC}
Let $n$ be a positive integer and let $0<\nu, \tau, \eta<1$ such that $4\sqrt[13]{\log^2 n/n}<\nu\leq \tau\leq \eta/16$.  If $G$ is a graph on $n$ vertices such that $\delta(G)\geq \eta n$ and $G$ is a $(\nu, \tau)$-robust expander, then $G$ has a Hamiltonian cycle.
\end{theorem}

\subsection{A simpler way to check that $G$ is a robust expander}

The main result of this subsection is to prove a general lemma which reduces the problem of checking that $G$ is a robust expander to the problem of showing that nearly independent sets robustly expand (c.f. Fact \ref{lem:indall}) and that $G$ has no sparse cuts.  Importantly, this will give us useful information in the case when $G$ is not a robust expander.  

First we establish a simple fact which roughly allows us to translate a bound on the  minimum degree to a bound on the size of a robust neighborhood.

\begin{fact}\label{robustdegree}
Let $\nu>0$, let $G$ be a graph on $n$ vertices, and let $S\subseteq V(G)$ with $|S| \geq (\sqrt{\nu}+\nu)n$.  Then
\begin{enumerate}
\item $|RN_{\nu}(S)| \geq \delta(S) - \sqrt{\nu}n$, and 
\item if there exists $S' \subseteq S$ with $|S'| \geq (\sqrt{\nu}+\nu) n$ such that $\delta(S')\geq |S|+(\sqrt{\nu}+\nu)n$, then
$|RN_{\nu}(S)| \geq |S| + \nu n $.
\end{enumerate}
\end{fact}

\begin{proof}
(i) For all $v\in RN_{\nu}(S)$, $\nu n\leq d(v, S)\leq |S|$, and for all $v\in V(G)\setminus RN_{\nu}(S)$, $d(v, S)<\nu n$.  Thus
\[ |S|\delta(S) \leq \sum_{v \in S} d(v) = \sum_{v \in V(G)} d(v, S) \leq |RN_{\nu}(S)||S| + \nu n(n - |RN_{\nu}(S)|). \]
Solving for $RN_{\nu}(S)$, we have
\begin{align*}
|RN_{\nu}(S)| \geq \frac{|S|\delta(S) - \nu n^2}{|S| - \nu n} = \delta(S) + \frac{\nu n \delta(S) - \nu n^2}{|S| - \nu n} 
\geq \delta(S) - \frac{n - \delta(S)}{\sqrt{\nu}n} \nu n \geq \delta(S) - \sqrt{\nu} n.
\end{align*}

(ii) Using the fact that $RN_\nu(S')\subseteq RN_\nu(S)$, we apply part (i) to get
\[ |RN_{\nu}(S)| \geq |RN_{\nu}(S')| \geq \delta(S') - \sqrt{\nu}n \geq |S| + \nu n. \qedhere\]
\end{proof}

\begin{lemma}\label{robustorindep}
Let $0<\nu \leq \tau^2\leq 1/4$ and let $G$ be a graph on $n$ vertices. If
\begin{enumerate}
\item $e(V(G)\setminus B, B) \geq 2\tau^2 n^2$ for all $B \subseteq V(G)$ with $\tau n\leq |B|\leq (1-\tau)n$ and
\item $|RN_{\nu}(S')| \geq |S'| + \nu n$ for all $S' \subseteq V(G)$ with $\tau^2 n \leq |S'| \leq (1-\tau^2) n$ and $\Delta(G[S']) < \nu^2 n$,
\end{enumerate}
then $G$ is a $(\nu^2, \tau)$-robust expander.
\end{lemma}

\begin{proof}
Suppose (i) and (ii) and let $S\subseteq V(G)$ with $\tau n\leq |S|\leq (1-\tau)n$.  Set $R':=RN_{\nu^2}(S)\setminus S$, $S':=S\setminus RN_{\nu^2}(S)$, and $T:=S\cap RN_{\nu^2}(S)$.  Suppose for contradiction that $$|RN_{\nu^2}(S)|<|S|+\nu^2n,$$ which with the defintions above implies  
\begin{equation}\label{R'}
|R'|<|S'|+\nu^2 n.  
\end{equation}

If $|S'|<\tau^2 n$, then \eqref{R'} implies
\begin{align*}
e(S, V(G)\setminus S) < (n - |S| - |R'|)\nu^2 n + |R'||S| <(n-|S|)\nu^2 n+(\tau^2 +\nu^2)n|S|\leq 2\tau^2 n^2,
\end{align*}
contradicting (i).

So suppose $|S'|\geq \tau^2 n$.  Note that $\tau^2 n \leq |S'|\leq |S| \leq (1-\tau)n\leq (1-\tau^2)n$  and $RN_{\nu}(S')\subseteq RN_{\nu^2}(S)$.  By the definition of $S'$ we have $d(v, S)<\nu^2 n$ for all $v\in S'$, and thus $e(S', T)<\nu^2 n|S'|\leq \nu^2 (1-\tau) n^2$.  So by averaging, 
\begin{equation}\label{RS'}
|RN_{\nu}(S')\cap T|< \nu (1-\tau) n.  
\end{equation}
By the definition of $S'$ we have $\Delta(G[S'])< \nu^2 n$, which by (ii), \eqref{R'}, and \eqref{RS'} implies
\begin{align*}
|S'|+\nu n\leq |RN_{\nu}(S')| \leq |R'| + |RN_{\nu}(S')\cap T| < |S'|+\nu^2 n + \nu(1-\tau) n,
\end{align*}
a contradiction.
\end{proof}

\subsection{Proof of Theorem \ref{l-boundedHC}}

To prove Theorem \ref{l-boundedHC}, we show that a $k$-partite graph satisfying the degree condition is either a robust expander (in which case we can apply Theorem \ref{robustexpanderHC}), or has a very specific extremal structure (in which case we can construct the Hamiltonian cycle directly).  We will make this precise below.

Since the results of this section are asymptotic, we ignore floors and ceilings whenever possible.  Formally, this means we will redefine $\Phi$ for the purposes of this section as follows.  For all $P=(n_1, n_2, \dots, n_k)\in \cP_k(n)$, let
\begin{itemize}
\item $f_i(P)=|V_i|+\sum_{j=1}^i|V_j|~ \text{ and } ~f(P)=\max_{i\in [\mu-1]}f_i(P) ,$

\item $g(P)=\frac{n}{2}+\frac{|V_\mu|}{2} ,$

\item $h_1(P)=\frac{n}{2}+|V_\lambda|,~ h_2(P)=\frac{n}{2}+\frac{n}{2\lambda}, \text{ and } h(P)=\min\{h_1(P), h_2(P)\} ,$

\item $\Phi(P)=\max\{f(P), g(P), h(P)\} .$
\end{itemize}
Note that we do not modify the definitions of $\lambda$ and $\mu$.

In light of Lemma \ref{robustorindep} and the plan set forth above, we first check that $G$ has no sparse cuts; i.e.\ condition (i) of Lemma \ref{robustorindep} is satisfied.  Note that we are able to show that a weaker degree condition suffices for this purpose.

\begin{proposition}\label{nosparsecuts}
Let $n\geq k \geq 2$ and $0 < \tau \leq \frac{\gamma}{4}$.  If $G=(V_1\sqcup \dots \sqcup V_k, E)$ is a $k$-partite graph on $n$ vertices with $n/2\geq |V_1|\geq \dots\geq |V_k|$ such that \[ \delta(V_i) \geq \max\left\{f(|V_1|, \dots, |V_k|), g(|V_1|, \dots, |V_k|)\right\} +\gamma n-|V_i| ~\text{ for all } i\in [k],\]
then for all $B \subseteq V(G)$ with $\tau n \leq |B| \leq (1-\tau)n$,
\[ e(V(G)\setminus B, B) \geq 2\tau^2 n^2 .\]
\end{proposition}

\begin{proof}
Let $B\subseteq V(G)$ with $\tau n\leq |B|\leq (1-\tau)n$, let $A := V(G)\setminus B$, $B_i := B \cap V_i$, and $A_i := A \cap V_i$.  Without loss of generality, suppose $|B| \leq \frac{n}{2}$.  Let $j$ be the largest index such that $$|B|+|A_j|>\max\left\{f, g\right\} +\tau n;$$ or if no such index exists, set $j=0$.  

If $j\geq \mu$, then since $|B|\leq n/2$ we have 
\begin{equation}\label{**}
|A_j| > g - |B| \geq \frac{|V_\mu|}{2} + \tau n, 
\end{equation}
and since $|V_j|\leq |V_\mu|$, we have 
\begin{equation}\label{*}
|A|+|B_j|= n+|V_j|-(|B|+|A_j|)<n+|V_j|-g \leq \frac{n}{2}+\frac{|V_\mu|}{2} = g.
\end{equation}
Thus 
\begin{align*}
e(A,B)\geq e(A_j, B)&\geq |A_j|\left(g+\gamma n-|V_j|-(|A|-|A_j|)\right)\\
&=|A_j|\left(g+\gamma n-(|A|+|B_j|)\right)\stackrel{\eqref{*}}{>} |A_j|\gamma n \stackrel{\eqref{**}}{>} \tau \gamma n^2 \geq 2\tau^2n^2.
\end{align*}

So suppose $0\leq j\leq \mu-1$, which by the definition of $j$ implies that 
\begin{equation}\label{***}
|B|+|A_i|\leq g + \tau n \text{ for all } i\geq j+1.  
\end{equation}
So if $\sum_{i=j+1}^k|B_i|=|B|-\sum_{i=1}^j|B_i|\geq \tau n$ (note that this necessarily holds when $j=0$), then  
\begin{align*}
e(A,B)= \sum_{i=j+1}^ke(A, B_i)&\geq \sum_{i=j+1}^k|B_i|\left(g+\gamma n-|V_i|-(|B|-|B_i|)\right)\\
&= \sum_{i=j+1}^k|B_i|\left(g+\gamma n-(|B|+|A_i|)\right)\\
&\stackrel{\eqref{***}}{\geq} \tau \frac{\gamma}{2} n^2 \geq 2\tau^2 n^2.
\end{align*}
Otherwise, we have $1\leq j\leq \mu-1$ and $|B|-\tau n< \sum_{i=1}^j|B_i|\leq \sum_{i=1}^j|V_i|$ which implies
\begin{align*}
|A_j|> f_j + \tau n - |B| = |V_j|+\sum_{i=1}^j|V_i|+\tau n-|B|> |V_j|,
\end{align*}
a contradiction.
\end{proof}

We now show that either all nearly independent sets robustly expand (c.f. Theorem \ref{warmup}); i.e.\ condition (ii) of Lemma \ref{robustorindep} is satisfied, or else we are in an extremal case.  We say that a $k$-partite graph $G$ is \tbf{$\nu$-extremal} if there exists $S\subseteq V(G)$ such that $\Delta(G[S]) < \nu^2 n$ and 
\begin{equation}\label{extremalS}
\frac{n}{2} - \nu n \leq |S| \leq  \sum_{i:V_i\cap S\neq \emptyset}|V_i|\leq \frac{n}{2}.
\end{equation}

\begin{proposition}\label{degcondrobust}
Let $n\geq k\geq 2$ and $\gamma, \tau, \nu, n_0$ such that $0 \leq \frac{1}{n_0}\ll\nu\leq \min\{\frac{4}{(k-1)^2}, \frac{\gamma^2}{4}, \frac{\tau^4}{4k^2}\}$.  If  $G=(V_1\sqcup \dots \sqcup V_k, E)$ is a $k$-partite graph on $n \geq n_0$ vertices with $n/2\geq |V_1|\geq \dots\geq |V_k|$ such that
\[ \delta(V_i) \geq \Phi(|V_1|, \dots, |V_k|) + \gamma n - |V_i| ~\text{ for all } i\in [k],\]
then $|RN_{\nu}(S)| \geq |S| + \nu n$ for all $S\subseteq V(G)$ with $\tau^2 n \leq |S| \leq (1-\tau^2)n$ and $\Delta(G[S]) < \nu^2 n$, or else $G$ is $\nu$-extremal.
\end{proposition}

\begin{proof}
Let $S\subseteq V(G)$ with $\tau^2 n \leq |S| \leq (1-\tau^2)n$ and $\Delta(G[S]) < \nu^2 n$ which implies $S\cap RN_\nu(S)=\emptyset$.  Since $S\cap RN_\nu(S)=\emptyset$, we have $|S|\leq n-|RN_\nu(S)|$ and thus
\begin{equation}\label{RN>1/2}
\text{if } |RN_\nu(S)|\geq (1+\nu)\frac{n}{2}, \text{ then } |RN_\nu(S)|\geq |S|+\nu n.
\end{equation}

Relabel the parts as $U_1, \dots, U_k$ and define $r$ so that $|U_i \cap S| \geq \frac{\tau^2}{k} n$ for $1 \leq i \leq r$ and $|U_i \cap S| < \frac{\tau^2}{k} n$ for $r+1 \leq i \leq k$, and $|U_1| \geq \cdots \geq |U_r|$ and $|U_{r+1}| \geq \cdots \geq |U_k|$.  Let $A := U_1 \cup \cdots \cup U_r$, $B := U_{r+1} \cup \cdots \cup U_k$, and $S_i := U_i \cap S$.

Note that $|S_r| \geq \frac{\tau^2}{k} n\geq (\nu+\sqrt{\nu})n$, 
so by Fact \ref{robustdegree}.(ii) we have 
\begin{equation}\label{RNS_r}
|RN_{\nu}(S)| \geq \delta(S_r)-\sqrt{\nu}n\geq \delta(U_r)-\sqrt{\nu}n.
\end{equation}

If $1\leq r\leq \mu-1$, then $|S|\leq \sum_{i=1}^r|U_i|\leq \sum_{i=1}^r|V_i|$, which by \eqref{RNS_r} implies 
\begin{align*}
|RN_\nu(S)|&\geq \delta(U_r)-\sqrt{\nu}n\\
&\geq \delta(V_r)-\sqrt{\nu}n
\geq f_r - |V_i| + \gamma n \geq \sum_{i=1}^r|V_i|+\nu n
\geq |S|+\nu n.
\end{align*}

So for the remainder of the proof we suppose $r\geq \mu$ which implies 
\[
|RN_{\nu}(S)| \geq \delta(U_r)-\sqrt{\nu}n\geq \delta(V_\mu)-\sqrt{\nu n}\geq g - |V_\mu| + \gamma n \geq \frac{n}{2}-\frac{|V_\mu|}{2}+\nu n.
\]
If $|S|+\nu n\leq \frac{n}{2}-\frac{|V_\mu|}{2}+\nu n\leq |RN_{\nu}(S)|$, then we are done; so suppose that 
\begin{equation}\label{S_Ur}
|S| > \frac{n}{2}-\frac{|V_\mu|}{2}.
\end{equation}
Let $v\in B$ and note that by \eqref{S_Ur} we have 
\begin{align*}
d(v,S)\geq g+\gamma n-|B|-(|A|-|S|)=|S|-\left(\frac{1}{2}-\frac{|V_\mu|}{2}-\gamma\right)n> \gamma n\geq \nu n.
\end{align*}
Thus
\begin{equation}\label{BRN}
B \subseteq RN_{\nu}(S),
\end{equation}
and since $\Delta(G[S])<\nu^2 n$, we have $S \cap U_i = \emptyset$ for all $r+1 \leq i \leq k$.

\textbf{Case 1} ($\mu\leq r\leq \lambda-1$).
Using the fact that $S\subseteq A$, we have  
$$|S|\leq |A|\leq \sum_{i=1}^r|U_i| \leq \frac{n}{2},$$
where the second inequality holds by Observation \ref{l-fairfacts}.\ref{2} (since $r\leq \lambda-1$).
Now if  $|RN_{\nu}(S)|<|S| + \nu n$, then by \eqref{BRN} we have
\[ |S| + \nu n > |RN_{\nu}(S)| \geq |B| = n - |A| \geq \frac{n}{2}. \]
Rearranging gives
\[ \frac{n}{2}-\nu n< |S| \leq \frac{n}{2}, \]
which together with the fact that $\Delta(G[S])< \nu^2 n$ implies that $G$ is $\nu$-extremal.

\textbf{Case 2} ($r \geq \lambda$).  If $h = h_1$, then \eqref{RNS_r} implies $$|RN_\nu(S)|\geq \delta(U_r)-\sqrt{\nu}n\geq \delta(V_\lambda)-\sqrt{\nu}n\geq h_1 - |V_\lambda| + \gamma n - \sqrt{\nu} \geq (1+\nu)\frac{n}{2}$$ and thus we are done by \eqref{RN>1/2}.  So suppose $h = h_2$, that is, $|V_\lambda| \geq \frac{n}{2\lambda}$.

For all $v\in S_1$ we have
\[d(v, A)\geq h_2 + \gamma n-|U_1|-|B|=  \sum_{i=2}^r\left(|U_i|-\frac{\lambda-1}{r-1}\frac{n}{2\lambda}\right)+\gamma n.\]
So by Fact \ref{robustdegree}.(i) and the bounds on $\nu$, we have
\begin{align}
|RN_{\nu}(S_1) \cap A| &\geq \sum_{i=2}^r\left(|U_i|-\frac{\lambda-1}{r-1}\frac{n}{2\lambda}\right)+\gamma n -\sqrt{\nu}n\notag\\
&\geq \sum_{i=2}^r\left(|U_i|-\frac{\lambda-1}{r-1}\frac{n}{2\lambda}\right)+(r-1)\frac{\nu n}{2}.\label{RNSi}
\end{align}
So by \eqref{RNSi} there exists some $i$ with $2\leq i\leq r$, such that $|RN_{\nu}(S_1)\cap U_i|\geq |U_i|-\frac{\lambda-1}{r-1}\frac{n}{2\lambda}+\frac{\nu n}{2}$.  Now using Fact \ref{robustdegree}.(i) again to get a lower bound on $|RN_{\nu}(S_i)|$ and using the fact that $r\geq \lambda$, we have
\begin{align*}
|RN_{\nu}(S)|&\geq |RN_{\nu}(S_1)\cap U_i|+| RN_{\nu}(S_i)|\\
&\geq |U_i|-\frac{\lambda-1}{r-1}\frac{n}{2\lambda}+\frac{\nu n}{2}+ h_2-|U_i|\geq (1+\nu)\frac{n}{2},
\end{align*}
so by \eqref{RN>1/2} we have $|RN_{\nu}(S)| \geq |S| + \nu n$.
\end{proof}

Now we show that if $G$ is $\nu$-extremal, then despite the fact that $G$ is not a robust expander, $G$ has a Hamiltonian cycle.  We use the following theorem which gives a degree sequence condition for a bipartite graph to be \emph{Hamiltonian-biconnected}.  

\begin{theorem}[Berge {\cite[Chapter 10, Theorem 14]{Ber}}]\label{berge_lemma}
Let $G=(U,V,E)$ be a bipartite graph on $2m\geq 4$ vertices with vertices in $U=\{u_1, \dots, u_m\}$ and $V=\{v_1, \dots, v_m\}$ such that $d(u_1) \leq \cdots \leq d(u_m)$ and $d(v_1) \leq \cdots \leq d(v_m)$.  If for the smallest two indices $j$ and $k$ such that $d(u_j)\leq j+1$ and $d(v_k)\leq k+1$ (if they exist), we have $$d(u_j)+d(v_k)\geq m+2,$$
then for all $u\in U$ and $v\in V$, there exists a Hamiltonian path having $u$ and $v$ as endpoints.
\end{theorem}

\begin{proposition}\label{extremal}
Let $n\geq k\geq 2$, $0< \frac{1}{n_0} < \nu < \min\{\frac{\gamma}{6}, \frac{1}{2\lambda}\}$.  If $G=(V_1\sqcup \dots \sqcup V_k, E)$ is a $k$-partite graph on $n\geq n_0$ vertices with $n/2\geq |V_1|\geq \dots\geq |V_k|$ such that $G$ is $\nu$-extremal and 
\[ \delta(V_i) \geq \Phi(|V_1|, \dots, |V_k|) - |V_i| + \gamma n ~\text{ for all } i\in [k],\]
then $G$ has a Hamiltonian cycle.
\end{proposition}

\begin{proof}
Suppose $G$ is $\nu$-extremal as witnessed by a set $S\subseteq V(G)$.  Let $U_1, \dots, U_r$ be the parts of $G$ which have non-empty intersection with $S$.  Let $A=\bigcup_{i=1}^rU_i$ and $B=V(G)\setminus A$.  
Let $0\leq \ceiling{|B|-\frac{n}{2}}=:t\leq \nu n$.   Let $\cB=\{U_i: U_i\subseteq B\}$, let $\cB^+=\{U_i\in \cB: |U_i|>h-\frac{n}{2}+\frac{\nu n}{\lambda}\}$, let $B^+=\bigcup_{U_i\in \cB^+}U_i$, and let $B^*=B\setminus B^+$.  Note that 
\begin{equation}\label{A-S}
|A\setminus S|\leq \nu n.
\end{equation} 
Also note that $$|B^+|=\sum_{U_i\in \cB^+}|U_i|\leq n/2;$$ indeed, if $h=h_1$, then by definition $V_\lambda\not\in \cB^+$ and thus $|B^+|=\sum_{U_i\in \cB^+}|U_i|\leq \sum_{i=1}^{\lambda-1}|V_i|\leq n/2$ and if $h=h_2$, then $|\cB^+|\leq \lambda-1$ (which implies $|B^+| \leq n/2$, both by Observation \ref{l-fairfacts}.\ref{2}), as otherwise 
$$|B|> |\cB^+|(\frac{n}{2\lambda}+\frac{\nu n}{\lambda})\geq \frac{n}{2}+\nu n.$$
Thus $|B^*|=\sum_{U_i\in \cB\setminus \cB^+}|U_i|\geq |B|-\frac{n}{2}$, which implies $|B^*|\geq \ceiling{|B|-\frac{n}{2}}=t$.  Now for all $v\in B^*$, using the fact that $|A|\leq n/2$, we have
\begin{equation}\label{B*}
d(v, B)\geq h+\gamma n- \left(h-\frac{n}{2}+\frac{\nu n}{\lambda}\right)-|A|\geq \frac{\gamma n}{2}.
\end{equation}
So we let $y_1, \dots, y_t\in B^*$.  By \eqref{B*} and the fact that $t\leq \nu n \leq \frac{\gamma}{6} n$, we can greedily choose distinct vertices $x_1, z_1, \dots, x_t, z_t\in B\setminus \{y_1, \dots, y_t\}$ such that $y_ix_i, y_iz_i\in E(G)$ for $i \in [t]$.  Note that for all $v\in B$,
\begin{align}
d(v, S)\geq \Phi+\gamma n- |B|-(|A|-|S|) \geq g-\frac{n}{2}+\frac{\gamma n}{2} \geq \frac{\gamma n}{2},\label{BtoS}
\end{align}
where the last inequality holds by \eqref{extremalS}.

By \eqref{BtoS} we can choose distinct vertices $a_1, a_1', a_2, a_2', \dots, a_{t-1}, a_{t-1}', a_t\in S$ such that $a_ix_i, a_i'z_i\in E(G)$ for all $i\in [t]$.  Now for all $u, v\in S$, by Observation \ref{l-fairfacts}.\ref{7} and \eqref{A-S}, we have that
\begin{align}
|(N(u)\cap N(v))\cap B|&\geq 2 \left( \Phi+\gamma n -|V_1|-\nu n-|A\setminus S| \right) -|B|\notag\\
&\geq \frac{n}{2} +\gamma n+\nu n-|B| \geq \gamma n \geq 6 \nu n \geq 5t\label{StoB}.
\end{align}
So by \eqref{StoB}, for all $a_i', a_{i+1}$ with $i\in [t-1]$ we can greedily choose a common neighbor $w_{i+1}\in B\setminus \{x_1,y_1,z_1, \dots, x_t,y_t,z_t, w_1,\dots, w_i\}$.  Let 
\[P^*=a_1x_1y_1z_1a_1'w_2a_2x_2y_2z_2a_2'w_3\dots w_{t}a_tx_ty_tz_t.\] If $n$ is even, set $P:=P^*$, and if $n$ is odd, set $P:=P^*-z_t$ (see Figure \ref{fig:P*}). 

\begin{figure}[ht]
\centering
\begin{tikzpicture}[scale = .7]
\draw (-.5, 0) rectangle (12.5, 3);
\draw[fill = lightgray!50] (-.4, 1.7) rectangle (12.4, 2.7);
\draw (-1.5, 1.5) node{$B$};
\draw (-.5, -.5) rectangle (12.5, -3.5);
\draw[fill = lightgray!50] (-.4, -.6) rectangle (12.4, -3);
\draw (-1.5, -2) node{$A$};

\draw (0, -1) node{$\bullet$} node[below]{$a_1$} --
(.5, .5) node{$\bullet$} node[above]{$x_1~~$} --
(1, 2) node{$\bullet$} node[above]{$y_1$} --
(1.5, .5) node{$\bullet$} node[above left]{$z_1$} --
(2, -1) node{$\bullet$} node[below]{$a_1'$} --
(2.5, .5)node{$\bullet$} node[above]{$w_2$} --
(3, -1) node{$\bullet$} node[below]{$a_2$} --
(3.5, .5) node{$\bullet$} node[above]{$x_2~~$} --
(4, 2) node{$\bullet$} node[above]{$y_2$} --
(4.5, .5) node{$\bullet$} node[above left]{$z_2$} --
(5, -1) node{$\bullet$} node[below]{$a_2'$} --
(5.5, .5) node{$\bullet$} node[above]{$w_3$} --
(6, -1) node{$\bullet$} node[below]{$a_3$};
\draw[dashed] (6, -1) -- (6.5, .5);
\draw (7.5, -.25) node{$\cdots$};
\draw (7.5, 2.3) node{$B^*$};
\draw (7.5, -2) node{$S$};
\draw[dashed] (8.5, .5) -- (9, -1);
\draw (9, -1) node{$\bullet$} node[below]{$a_{t-1}'$} --
(9.5, .5) node{$\bullet$} node[above]{$w_t$} --
(10, -1) node{$\bullet$} node[below]{$a_t$} --
(10.5, .5) node{$\bullet$} node[above]{$x_t~~$} --
(11, 2) node{$\bullet$} node[above]{$y_t$} --
(11.5, .5) node{$\bullet$} node[above left]{$z_t$};
\end{tikzpicture}
\caption{The path $P^*$.}\label{fig:P*}
\end{figure}

We now claim that in either case $|A\setminus P|=|B\setminus P|$.
By the definition of $t$, we have $|B|-|A|=2t$ if $n$ is even and $|B|-|A|=\frac{n-1}{2}+t-(\frac{n+1}{2}-t)=2t-1$ if $n$ is odd.  If $n$ is even, we have 
$$|A\setminus P|=|A|-2t-1=|B|-4t-1=|B\setminus P|,$$ and if $n$ is odd we have
$$|A\setminus P|=|A|-2t-1=|B|-2t-1-(2t-1)=|B|-(4t-2)=|B\setminus P|.$$  

Let $G'$ be the graph obtained from $G[A, B]$ by deleting all the interior vertices of $P$. Note that by the calculations above and since the endpoints of $P$ are in different parts, $G'$ is a balanced bipartite graph. Let $A'$ and $B'$ be the parts of $G'$ with $A' \subseteq A$ and $B' \subseteq B$ and $|A'|=n'$ and $|B'|=n'$. Additionally, let $S' = A' \cap S$, and for all $i\in [k]$ let $S_i'=S'\cap V_i$. We will use Theorem \ref{berge_lemma} to show that $G'$ is Hamiltonian bi-connected, which will allow us to find a Hamiltonian path in $G'$ connecting the endpoints of $P$.

We have
\begin{align}
\delta(B', A') &\geq \Phi+\gamma n - |B| - 2t \geq \Phi-\frac{n}{2} + 1,\label{e1} \\
\delta(A', B') &\geq \Phi+\gamma n - |A| - 4t \geq \frac{\gamma n}{2},\label{e2}
\end{align}
and for all $h\in [\mu-1]$ and $i\in [k]\setminus [\mu-1]$, 
\begin{align}
\delta(S_h', B') &\geq f_h(G)+\gamma n - |V_h|-\nu n-|A\setminus S| - 4t> \sum_{i=1}^h|V_i| +1, \label{e3}\\
\delta(S_i', B') &\geq g(G)+\gamma n - |V_\mu|-\nu n-|A\setminus S| - 4t\notag\\
&\geq \frac{n}{2}-\frac{|V_\mu|}{2} +1\geq \frac{n}{2}-\left(\Phi-\frac{n}{2}\right)+1=n-\Phi+1. \label{e4}
\end{align}

Let $a_1, \dots, a_m$ be the vertices of $A'$ and let $b_1, \dots, b_m$ be the vertices of $B'$, with $d_{G'}(a_1) \leq \cdots \leq d_{G'}(a_m)$ and $d_{G'}(b_1) \leq \cdots \leq d_{G'}(b_m)$, and let $j$ and $k$ be the smallest two indices such that $d_{G'}(a_j) \leq j+1$ and $d_{G'}(b_k) \leq k+1$.  By \eqref{e2}, since 
\[|A'\setminus S'| = |A\setminus S|\leq  \nu n< \frac{\gamma n}{2} \leq \delta(A', B'),\] 
we have $a_j \in S'$.  From \eqref{e3}, we have $a_j\not\in S_h'$ for all $h<\mu$, so $a_j\in S_i'$ for some $i\geq \mu$, and thus by \eqref{e1} and \eqref{e4} we have
\[ d_{G'}(a_j) + d_{G'}(b_k) \geq \delta(S_i', B') + \delta(B', A') \geq \frac{n}{2} + 2 \geq n' + 2.\]
So by Theorem \ref{berge_lemma}, $G'$ has a Hamiltonian path between any $a\in A'$ and $b\in B'$. In particular, there is a Hamiltonian path $H$ between the two endpoints of $P$, so our desired Hamiltonian cycle in $G$ is $H \cup P$.

\end{proof}

Finally, we put everything together and prove that for all $n\geq k\geq 2$ and $0< \frac{1}{n_0}\ll \gamma$, if $G=(V_1\sqcup \dots \sqcup V_k, E)$ is a $k$-partite graph on $n \geq n_0$ vertices with $n/2\geq |V_1|\geq \dots\geq |V_k|$ such that
\begin{equation}\label{phi}
\delta(V_i) \geq \Phi +\gamma n-|V_i| ~\text{ for all } i\in [k],
\end{equation}
then $G$ has a Hamiltonian cycle.

\begin{proof}[Proof of Theorem \ref{l-boundedHC}]
Let $0<\tau \leq \frac{\gamma}{4}$ and $0< \frac{1}{n_0}\ll\nu \leq \min\{\frac{4}{(k-1)^2}, \frac{\tau^4}{4k^2}\}\leq \min\{\frac{\gamma^2}{4}, \frac{\gamma}{2\lambda}, \frac{\gamma}{6}\}$.  Suppose $G$ satisfies \eqref{phi}.
By Proposition \ref{degcondrobust}, either $G$ is $\nu$-extremal, in which case we are done by Proposition \ref{extremal}, or else for all $S'\subseteq V(G)$ with $\tau^2n\leq |S'|\leq (1-\tau^2)n$ and $\Delta(G[S'])<\nu^2 n$, we have $|RN_\nu(S')|\geq |S'|+\nu n$.  From Lemma \ref{nosparsecuts}, for all $B\subseteq V(G)$ with $\tau n\leq |B|\leq (1-\tau)n$, we have $e(B, V(G)\setminus B)\geq 2\tau^2 n^2$.  So by Lemma \ref{robustorindep}, $G$ is a $(\nu^2, \tau)$-robust expander.  Finally, since $\delta(G)\geq n/4$ (by Observation \ref{l-fairfacts}.\ref{7}) and $G$ is a $(\nu^2, \tau)$-robust expander, $G$ has a Hamiltonian cycle by Theorem \ref{robustexpanderHC}.
\end{proof}

\section{Acknowledgements}

We thank the referees for their suggestions which helped improve the exposition.

\end{document}